%% file: CDC2025.tex
\documentclass[letterpaper, 10 pt, conference]{ieeeconf}  

\IEEEoverridecommandlockouts                              
\input{Notation/packages}

\input{Notation/general_commands}

\input{Notation/abbreviations}

\input{Notation/notation_macpomdp}
\input{Notation/notation_ciapproach}

\allowdisplaybreaks

\title{\LARGE \bf
On Equivalence Between Decentralized Policy-Profile Mixtures and Behavioral Coordination Policies in Multi-Agent Systems
}

\author{Nouman Khan and Vijay G. Subramanian
\thanks{N. Khan and V. G. Subramanian are with the EECS Department, University of Michigan, Ann Arbor, MI, 48109-2122. This work was partially supported by NSF Awards CNS 1955777, CCF 2008130, CMMI 2240981, and ECCS 2038416, and ONR Award No. N000142412742. Emails: {\tt\small \{knouman,vgsubram\}@umich.edu}.
        }%
}

\begin{document}

\maketitle
\thispagestyle{empty}
\pagestyle{empty}

\begin{abstract}
Constrained decentralized team problem formulations are good models for many cooperative multi-agent systems. Constraints necessitate randomization when solving for optimal solutions---past results show that joint randomization in the team is in general necessary for (strong) Lagrangian duality to hold---, but a better understanding of randomization still remains. 
For a partially observed multi-agent system with a Borel hidden state, countable observations, and finite actions, we prove the following: \textit{i}) independently randomized decentralized policy-profiles---whether behavioral or pure---induce the same occupation measures (on joint-history and joint-action pairs) as decentralized behavioral policy-profiles; and \textit{ii}) jointly randomized behavioral and pure decentralized policy-profiles induce the same occupation measures. Restricting to finite observations, we also prove that joint mixtures of decentralized policy-profiles (both pure and behavioral) and common information based behavioral coordination policies (also mixtures of them) induce the same occupation measures. This generalizes past work that shows equivalence between pure decentralized policy-profiles and pure coordination policies. These results can be used to develop further results on Lagrangian duality, the minimum number of randomizations needed in an optimal behavioral coordination policy, and learning based schemes that can find approximately optimal solutions.

\end{abstract}

\input{Sections/introduction}
\input{Sections/model}
\input{Sections/decentralizedCase}
\input{Sections/coordinatedCase}
\input{Sections/conclusion}

\addtolength{\textheight}{-3cm}   



\bibliographystyle{ieeetr} 
\bibliography{ref_macpomdp, ref_macvs, ref_sa}

\end{document}

%% file: Notation/packages.tex
\usepackage{amsmath}
\usepackage{amsfonts}
\usepackage{amsthm}
\usepackage{amssymb}
\usepackage{svg}
\usepackage{mathrsfs}
\usepackage{dsfont}
\usepackage{comment}
\usepackage{epstopdf}
\usepackage{epsf,epsfig,multirow}
\usepackage{hhline}
\usepackage{mathabx}
\usepackage{physics}
\usepackage{bbm}
\usepackage[normalem]{ulem}
\let\labelindent\relax
\usepackage{enumitem}
\usepackage{xcolor}
\usepackage{ulem}
\usepackage{cite}
\usepackage{url}

\usepackage{array}
\usepackage{algorithm}
\usepackage[noend]{algpseudocode}

\usepackage{accents}
\usepackage{longtable}
\usepackage{tabularray}
\usepackage{lmodern}
\usepackage{tensor}
\usepackage{color-edits}
\addauthor{nk}{blue}
\addauthor{vj}{red}
\usepackage{tikz}

%% file: Notation/general_commands.tex
\newtheorem{thm}{Theorem}
\newtheorem{prop}{Proposition}
\newtheorem{lem}{Lemma}
\newtheorem{cor}{Corollary}

\newtheorem{assumption}{Assumption}

\theoremstyle{definition}
\newtheorem{dfn}{Definition}

\theoremstyle{remark} 
\newtheorem{rem}{\textit{Remark}}

\allowdisplaybreaks 
\global\long\def\11{\mathds{1}}
\allowdisplaybreaks 

\newcommand\numberthis{\addtocounter{equation}{1}\tag{\theequation}}
\newcommand{\ra}{\rightarrow}
\newcommand{\la}{\leftarrow}

\newcommand{\mcl}{\mathcal}
\newcommand{\mbb}{\mathbb}
\newcommand{\mbf}{\mathbf}
\newcommand{\mrm}{\mathrm}
\newcommand{\eps}{\epsilon}
\newcommand{\ov}{\overline}
\newcommand{\udl}{\underline}

\def \defeq{:=}
\def \defeqr{=:}

\newcommand{\wh}{\widehat}
\newcommand{\wtd}{\widetilde}
\def \l{\left}
\def \r{\right}
\def \pr{\mbb{P}}

\newcommand{\dotp}[2]{\l\langle #1 \; , \; #2 \r\rangle}

\newcounter{relctr} 
\everydisplay\expandafter{\the\everydisplay\setcounter{relctr}{0}} 

\newcommand\labelrel[2]{%
	\begingroup
	\refstepcounter{relctr}%
	\stackrel{\textnormal{(\alph{relctr})}}{\mathstrut{#1}}%
	\originallabel{#2}%
	\endgroup
}
\AtBeginDocument{\let\originallabel\label} 

\usepackage[linesnumbered,ruled,vlined,algo2e]{algorithm2e}

\SetCommentSty{mycommfont}
\let\oldnl\nl
\newcommand{\nonl}{\renewcommand{\nl}{\let\nl\oldnl}}
\SetKwInput{KwInput}{Input}                
\SetKwInput{KwOutput}{Output}              
\SetKwInOut{Parameter}{Parameter}
\SetKwRepeat{Do}{do}{while}%


%% file: Notation/abbreviations.tex
\newcommand{\macpomdp}{\texttt{MA-C-POMDP}}

\newcommand{\sacmdp}{\texttt{SA-C-MDP}}

\newcommand{\mapomdp}{\texttt{MA-POMDP}}



%% file: Notation/notation_macpomdp.tex
\newcommand{\qt}[2]{{#1}_{#2}}
\newcommand{\qn}[2]{{#1}^{\l(#2\r)}}
\newcommand{\qtn}[3]{{#1}_{#2}^{\l(#3\r)}}
\newcommand{\qseq}[2]{ \hspace{2pt}^{#2}\hspace{-1pt}{#1}}
\newcommand{\indevent}[1]{ \11\l[ #1 \r] }
\newcommand{\onevector}{ \mbf{1} }

\newcommand{\tuplemacpomdp}{\mcl{J}}
\newcommand{\coptn}[2]{ \textit{\texttt{COP}}_{#1}^{(#2)}\! }

\newcommand{\Stt}[1]{ \qt{S}{#1} }
\newcommand{\stt}[1]{ \qt{s}{#1} }
\newcommand{\sspace}{\mcl{S}}

\newcommand{\Ot}[1]{ \qt{O}{#1} }

\newcommand{\Otn}[2]{ \qtn{O}{#1}{#2} }
\newcommand{\ot}[1]{ \qt{o}{#1} }

\newcommand{\otn}[2]{ \qtn{o}{#1}{#2} }
\newcommand{\ospace}{\mcl{O}}

\newcommand{\onspace}[1]{  \qn{\ospace}{#1} }
\newcommand{\otnspace}[2]{  \qtn{\ospace}{#1}{#2} }

\newcommand{\At}[1]{ \qt{A}{#1} }

\newcommand{\Atn}[2]{ \qtn{A}{#1}{#2} }
\newcommand{\at}[1]{ \qt{a}{#1} }
\newcommand{\an}[1]{ \qn{a}{#1} }
\newcommand{\atn}[2]{ \qtn{a}{#1}{#2} }
\newcommand{\aspace}{\mcl{A}}
\newcommand{\anspace}[1]{ \qn{\aspace}{#1} }
\newcommand{\atnspace}[2]{ \qtn{\aspace}{#1}{#2} }

\newcommand{\borel}[1]{\mcl{B}\l(#1\r)}

\newcommand{\Mta}[2]{ \qt{M}{#1}\l(#2\r) }

\newcommand{\Hst}[1]{ \qt{H}{#1} }

\newcommand{\Hstn}[2]{ \qtn{H}{#1}{#2} }
\newcommand{\hst}[2]{ \qt{#1}{#2} }
\newcommand{\hsn}[2]{ \qn{#1}{#2}  }
\newcommand{\hstn}[3]{ \qtn{#1}{#2}{#3} }
\newcommand{\hsspace}{ \mcl{H} }
\newcommand{\hstspace}[1]{ \qt{\hsspace}{#1} }
\newcommand{\hsnspace}[1]{ \qn{\hsspace}{#1} }
\newcommand{\hstnspace}[2]{ \qtn{\hsspace}{#1}{#2} }

\newcommand{\ut}[1]{ \qt{u}{#1} }
\newcommand{\un}[1]{ \qn{u}{#1}  }
\newcommand{\utn}[2]{ \qtn{u}{#1}{#2} }
\newcommand{\uspace}{\mcl{U}}
\newcommand{\utspace}[1]{\qt{\uspace}{#1}}
\newcommand{\unspace}[1]{ \qn{\uspace}{#1} }
\newcommand{\utnspace}[2]{ \qtn{\uspace}{#1}{#2}  }
\newcommand{\uspacepure}{ \qt{\uspace}{\mrm{pure}}}
\newcommand{\utspacepure}[1]{ \qt{\uspace}{\mrm{pure}, #1}}
\newcommand{\unspacepure}[1]{ \qtn{\uspace}{\mrm{pure}}{#1}}
\newcommand{\utnspacepure}[2]{ \qtn{\uspace}{\mrm{pure}, #1}{#2}}

\newcommand{\prup}[2]{ \qtn{\pr}{#2}{#1} }
\newcommand{\E}[2]{ \qtn{\mbb{E}}{#2}{#1} }
\newcommand{\tspace}{\mcl{T}}  

\newcommand{\constraintv}{ \breve{D} }	

\newcommand{\infsup}[2]{\inf\limits_{#1}\sup\limits_{#2}}
\newcommand{\supinf}[2]{\sup\limits_{#1}\inf\limits_{#2}}

\newcommand{\ocm}{\mbf{p}}

\newcommand{\ocmtn}[2]{ \qtn{\ocm}{#1}{#2} }

\newcommand{\cb}{\wh{c}}

\newcommand{\boldcbt}[1]{ \qt{ \wh{\mbf{c}}}{#1} }

\newcommand{\db}{\wh{d}}
\newcommand{\dbt}[1]{\qt{\db}{#1}}

\newcommand{\bolddbt}[1]{ \qt{\wh{\mbf{d}}}{#1} }

\newcommand{\cCost}{c \l( \Stt{t}, \At{t} \r)}

\newcommand{\dCost}{ d \l( \Stt{t}, \At{t} \r) }

\newcommand{\dtCost}[1]{ \qt{d}{#1} \l( \Stt{t}, \At{t} \r) }

\newcommand{\Ctn}[2]{\qtn{C}{#1}{#2}}
\newcommand{\ftn}[2]{\qtn{f}{#1}{#2}}
\newcommand{\Dtn}[2]{\qtn{D}{#1}{#2}}
\newcommand{\Ctavg}[1]{\qtn{C}{#1}{\mrm{avg}}}
\newcommand{\Dtavg}[1]{\qtn{D}{#1}{\mrm{avg}}}
\newcommand{\optcosttn}[2]{ \qtn{\udl{C}}{#1}{#2} }

\newcommand{\Ltn}[2]{ \qtn{L}{#1}{#2} }
\newcommand{\lagspace}{\mbb{R}^K_{\ge 0}}

\newcommand{\xtspace}[1]{ \mcl{X}_{#1} }

\newcommand{\equivalent}{\equiv}
\newcommand{\equivt}[1]{\qt{\equivalent}{#1}}

\newcommand{\dom}{\succeq}
\newcommand{\domt}[1]{\qt{\succeq}{#1}}

%% file: Notation/notation_ciapproach.tex
\newcommand{\wtf}{\wtd{f}}

\newcommand{\wtftn}[2]{\qtn{\wtf}{#1}{#2}}

\newcommand{\Gammat}[1]{ \qt{\Gamma}{#1} }

\newcommand{\Gammatn}[2]{ \qtn{\Gamma}{#1}{#2} }
\newcommand{\gammat}[1]{ \qt{\gamma}{#1} }

\newcommand{\gammatn}[2]{ \qtn{\gamma}{#1}{#2} }
\newcommand{\gammaspace}{\mcl{G}}
\newcommand{\gammatspace}[1]{  \qt{\gammaspace}{#1} }

\newcommand{\gammatnspace}[2]{ \qtn{\gammaspace}{#1}{#2} }

\newcommand{\wtHstn}[2]{ \qtn{\wtd{H}}{#1}{#2} }
\newcommand{\wthsn}[1]{ \qn{\wtd{h}}{#1} }
\newcommand{\wthstn}[2]{ \qtn{\wtd{h}}{#1}{#2} }
\newcommand{\wthstnspace}[2]{ \qtn{\wtd{\mcl{H}}}{#1}{#2} }

\newcommand{\vt}[1]{ \qt{v}{#1} }

\newcommand{\vvspace}{\mcl{V}}
\newcommand{\vtspace}[1]{\qt{\vvspace}{#1}  }
\newcommand{\vvspacepure}{ \qt{\vvspace}{\mrm{pure}} }
\newcommand{\vtspacepure}[1]{ \qt{\vvspace}{\mrm{pure}, #1} }

\newcommand{\ocmb}{\wtd{\mbf{p}}}
\newcommand{\ocmbtn}[2]{ \qtn{\ocmb}{#1}{#2} }

\newcommand{\cc}{\wtd{c}}

\newcommand{\boldcct}[1]{ \qt{ \wtd{\mbf{c}}}{#1} }

\newcommand{\dct}[1]{\qt{\wtd{d}}{#1}}

\newcommand{\bolddct}[1]{ \qt{ \wtd{\mbf{d}}}{#1} }


%% file: Sections/introduction.tex
\section{Introduction}\label{sec:intro}
Sequential decision making for cooperative multi-agent systems has been studied extensively in \emph{stochastic control}~\cite{witsenhausen1979structure,WalrandV83,witsenhausen1973standard,nayyarMT2011,nayyar13}, with the focus on \mapomdp s (Multi-Agent POMDPs) where all agents in the team optimize a single (long-term) cost. In this context, using the common information approach, \cite{nayyarMT2011,nayyar13} solve for the optimal control via a coordinator who knows only the common information of all the agents. There, a core structural result shows equivalence between pure (deterministic) decentralized policy-profiles and pure coordination profiles for systems with a finite number of states. 
This paper investigates the more general situation where one long-term (objective) cost must be minimized, while maintaining multiple other long-term (constraint) costs within prescribed limits. Such considerations arise in several real-world applications---such as communication networks, traffic management, energy-grid optimization, e-commerce pricing, environmental monitoring, etc.---where efficient operation must be balanced with maintaining safe operating margins. For this, we study a \textit{cooperative \textbf{M}ulti-\textbf{A}gent \textbf{C}onstrained \textbf{P}artially \textbf{O}bservable \textbf{M}arkov \textbf{D}ecision \textbf{P}rocess}, or simply (a cooperative) \macpomdp. 

In contrast to single-agent unconstrained decision problems~\cite{kumar2015stochastic}, in single-agent constrained decision problems~\cite{altman-constrainedMDP,borkar1988} randomized policies can outperform pure policies. This comparison between randomized and pure policies is well understood for decentralized (dynamic) teams, and also in (dynamic) games~\cite{anantharam2007common,dengwang2023dgaa} (where randomization is necessary for an equilibrium to exist, in general). However, detailed comparison of the different possibilities for randomization in multi-agent systems---jointly across all agents on the policy-profile space, independently by each agent on her own policy space, and using behavioral strategies that involve joint or independent randomizations over the joint-action space---is lacking. 
In this work, we contribute to filling this gap in the team setting as follows:
\begin{itemize}
	\item[\textit{i})] For the general model in Section~\ref{sec:macpomdp}, Section~\ref{sec:dec_pps} shows that independently randomized decentralized policy-profiles---whether behavioral or pure---induce the same set of occupation measures (defined in Section~\ref{sec:macpomdp}) as decentralized behavioral policy-profiles (Theorem~\ref{thm:dec_pps:ocm}). In the same setting, Section~\ref{sec:dec_pps} further shows that jointly randomized behavioral and pure decentralized policy-profiles induce the same set of occupation measures (Theorem~\ref{thm:dec_pps:ocm_3}). 
	
	\item[\textit{ii})] Under the additional assumption of finite number of observations, Section~\ref{sec:ciapproach} establishes that the occupation measures induced by joint mixtures of decentralized policy-profiles coincide with those induced by (common information based) behavioral coordination policies (Theorem~\ref{thm:ciapproach:equivalence}).
\end{itemize}

%% file: Sections/model.tex
\section{System Model and Preliminaries}\label{sec:macpomdp}

\subsection{System Model}
The sequential decision making model that we consider can be specified by a tuple $\tuplemacpomdp = \l( N, \sspace, \ospace, \aspace, \mcl{P}_{tr}, c, d, P_1 \r)$. Here, $N \in \mbb{N}$ denotes the number of agents (which is assumed fixed), $\sspace$ represents the state space, $\ospace$ is the joint-observation space, and $\aspace$ is the joint-action space. The transition-law is given by $\mcl{P}_{tr}$ while $P_1$ represents probability law for the initial state and joint-observation. The immediate costs are given by $c$ and $d$. These attributes and the decision process are described below.
\begin{itemize}[
	leftmargin=0pt, 
	itemindent=10pt,
	labelwidth=0pt, 
	labelsep=5pt, 
	]
	\item \textbf{State Process}: 
	The state space $\sspace$ is a topological space with the Borel $\sigma$-algebra $\borel{\sspace}$. The state process is denoted by $\l\{\Stt{t} \r\}_{t=1}^{\infty} \subseteq \sspace$, where $\Stt{t}$ denotes the state of the system at time-step $t \in \mbb{N}$.

	\item \textbf{Joint-Observation Process}: The joint-observation space $\ospace$ is a countable discrete space $ \ospace = \onspace{0:n}$, where $\onspace{0}$ denotes the common observation space of the team, and for each $n \in [N]$, $\onspace{n}$ denotes the private observation space of agent $n$. The joint-observation process is denoted by $\l\{ \Ot{t} \r\}_{t=1}^{\infty} \subseteq \ospace$, where $\Ot{t} = \Otn{t}{0:N}$ and is such that at time-step $t$, agent $n$ only observes $\Otn{t}{0}$ and $\Otn{t}{n}$.
	
	\item \textbf{Joint-Action Process}: The joint-action space $\aspace $ is a finite discrete space of the form $ \aspace  = \anspace{1:N} $, where $\anspace{n}$ denotes the action space of agent $n$. The joint-action process is denoted by $\l\{ \At{t} \r\}_{t=1}^{\infty} \subseteq \aspace$, where $\At{t} = \Atn{t}{1:N}$ and $\Atn{t}{n}$ denotes the action of agent $n$ at time-step $t$. Since all $\anspace{n}$'s (hence, $\aspace$ as well) are finite, they are all compact metric spaces (so, also complete and separable).
	
	\item \textbf{Transition-law}: 
	At time-step $t$, given the current state $\Stt{t}$ and the joint action $\At{t}$, the next state $\Stt{t+1}$ and joint observation $\Ot{t+1}$ evolve in a time-homogeneous manner independent of all previous states, all previous and current joint observations, and all previous joint actions. The \emph{transition-law}, $\mcl{P}_{tr}$, governing this evolution is given by a family of transition probabilities as follows:
	\begin{align*}
		\hspace{-1.1em}\mcl{P}_{tr} \!\defeq\! \l\{ P\l(  s,a,B,o \r) \!:\! s \in \sspace, a\in\aspace, B \in  \borel{\sspace}, o \in \ospace \r\},
	\end{align*}
	where, for all $B\in\borel{\sspace}$, $P\l( \cdot, \cdot, B, \cdot \r) : \sspace \times \aspace \times \ospace \ra [0,1]$ is Borel measurable, and for any $t \in \mbb{N}$, 
	\begin{align*}
		\begin{split}
			&\pr\big( \Stt{t+1} \in B , \Ot{t+1} = o | \Stt{1:t-1}  = \stt{1:t-1}, \Ot{1:t} = \ot{1:t}, \\
            &\hspace{10em} \At{1:t-1} = \at{1:t-1}, \Stt{t} = s, \At{t} = a \big) \\
			&\hspace{0pt}
			= \pr\l( \Stt{t+1} \! \in B, \Ot{t+1} \! = o | \Stt{t} = s, \At{t} = a \r)	
            \hspace{0pt} \defeqr \! P\l(  s,a,B,o \r).
		\end{split}
	\end{align*}

	\item \textbf{Initial Distribution}: $P_1$ is a (fixed) probability measure for the initial state and joint-observation pair, i.e., $P_1 \in \Mta{1}{\sspace\times \ospace}$, and for all $B \in \borel{S}$, 
	\begin{align*}
			\pr \l( \Stt{1} \in B, \Ot{1} = o \r) &\defeqr P_1\l( B, o \r).
	\end{align*}    
	
	\item \textbf{Immediate-costs}: 
	Let $c: \sspace \times \aspace \mapsto \mbb{R} $ be a Borel measurable scalar function, and $d: \sspace \times \aspace \mapsto \mbb{R}^K$ a Borel measurable vector function. Agents would like to minimize the expected discounted aggregate (defined later) of $c$ while keeping the expected discounted aggregate of $d$ below a specified threshold. We refer to $c$ as the \textit{immediate objective cost} and to $d$ as the \textit{immediate constraint cost}. Henceforth, both immediate costs are assumed to be bounded from below, i.e., Assumption~\ref{assmp:macpomdp:boundedcosts}(\textit{a}) always holds.
	\begin{assumption}[Boundedness of immediate costs] \label{assmp:macpomdp:boundedcosts}
        \leavevmode
		\begin{enumerate}[
			leftmargin=*,           
			labelsep=*,          
			labelwidth=!,           
			itemindent=0pt,         
			align=left              
			]		
			\item[(a)] Immediate costs are bounded from below: there exist $\udl{c} \in \mbb{R}$ and $\udl{d} \in \mbb{R}^{K} $ s.t. 
			$\udl{c} \le c(\cdot, \cdot)$ and $\udl{d} \le d( \cdot, \cdot)$.


			
			\item[(b)] Immediate costs are bounded from above: there exist $\ov{c} \in \mbb{R}$ and $\ov{d} \in \mbb{R}^K$ s.t. $c(\cdot, \cdot)  \le \ov{c}$ and $d(\cdot, \cdot) \le \ov{d}$.
		\end{enumerate}
	\end{assumption}
	
%
%
	
	\item \textbf{Decision Process}: 
	The \emph{decision process} proceeds as follows: 
	\textit{i}) At time-step $t$, the current state $\Stt{t}$ and joint-observation $\Ot{t}$ are generated (according to $\mcl{P}_{tr}$, or $P_1$ if $t=1$); \textit{ii}) each agent $n$ chooses her action $\Atn{t}{n}$; \textit{iii}) the immediate costs $c\l( \Stt{t}, \At{t} \r), d\l(\Stt{t}, \At{t}\r)$ are incurred; and \textit{iv}) the system moves to the next time-step, i.e., $t \la t+1$. 
	It is assumed that the agents take actions according to a (valid) causal \emph{interaction-mechanism}. 

	\begin{dfn}[Causal Interaction Mechanism]\label{dfn:macpomdp:interaction_mechanism}
		
		We say that agents follow a (valid) causal {interaction-mechanism} $x$ if for each time-step $t$, the joint-action $\At{t}$ is taken using knowledge of observations and actions up to that time, namely $\Ot{1:t-1} \At{1:t-1} $, and if it induces a unique probability measure, say $\prup{x}{P_1} = \prup{x}{P_1, \mcl{P}_{tr}}$, on the set of all trajectories of the process $\l\{ \l(\Stt{t}, \Ot{t}, \At{t}\r) \r\}_{t=1}^{\infty}$, namely $\l(\sspace \times \ospace \times \aspace\r)^{\infty} $ endowed with the Borel $\sigma$-algebra.

		We refer to the tuple $\Ot{1:t}\At{1:t-1}$ as the \emph{joint-history} of the team at time-step $t$ and denote it by $\Hst{t}$. The set of all realizations of $\Hst{t}$ is denoted by $\hstspace{t}$.
	\end{dfn}
	
\end{itemize}

\subsection{Constrained Optimization Problem}\label{sec:macpomdp:optimization_problem}
\subsubsection{Long-term Costs} 
Let $x$ denote a (valid) causal interaction-mechanism with a unique probability measure $ \prup{x}{P_1} = \prup{x}{P_1, \mcl{P}_{tr}}$ as mentioned above, and let $\E{x}{P_1}=\E{x}{P_1, \mcl{P}_{tr}}$ be the corresponding expectation operator. Using the \emph{discounted cost criterion} which de-emphasizes the far future, the long-term objective and constraint costs for time-horizon $T$ and discount-factor $\alpha$, such that $(T, \alpha) \in \l(\mbb{N} \cup \{\infty\}\r) \times (0,1] \setminus \{(\infty, 1)\}$, are defined as follows:
\begin{align*}
	\begin{split}
		\Ctn{T}{\alpha} \l( x \r) &\defeq  \E{x}{P_1} \l[ \sum_{t=1}^{T} \alpha^{t-1} \cCost \r] \text{, and } \\
		\Dtn{T}{\alpha}\l(x\r) & \defeq  \E{x}{P_1}\l[ \sum_{t=1}^{T} \alpha^{t-1} \dCost \r].
	\end{split}
\end{align*}

\subsubsection{Optimization Problem}
Let $X$ denote a set of (valid) causal interaction-mechanisms. Given discount-factor $\alpha$ and time-horizon $T$, the agents' collective goal as a team is to solve the following constrained optimization problem:
\begin{align}
	\begin{split}
		&\text{minimize } \Ctn{T}{\alpha}(x) \notag\\
		&\text{subject to } x \in X \text{ and } \Dtn{T}{\alpha}(x) \le \constraintv.
	\end{split}
	\Bigg\}
	\tag{$\coptn{T}{\alpha}\l( X \r)$-1}
	\label{eq:macpomdp:cop_1}
\end{align}
Here, $\constraintv$ is a fixed $K$-dimensional real-valued vector where $K$ is the number of constraints. 
The study of \eqref{eq:macpomdp:cop_1}---e.g., strong duality and existence of saddle-point---is \udl{not the focus} of this paper, and the reader is referred to \cite{khan2023,khan2023cooperative, nkhanthesis2025} for such results. 

\subsection{Preliminaries}\label{sec:macpomdp:preliminaries}
In this paper, we consider different choices of $X$ in \eqref{eq:macpomdp:cop_1}, and compare the sets of occupation measures that result. 

\begin{assumption}[Structural Assumptions]
\label{assmp:macpomdp:structured_X}
	\leavevmode
	\begin{enumerate}
		\item[(a)] 
		$X$ is endowed with a topology under which it forms a compact metric space.
		\item[(b)] 
		Under the topology mentioned in (a), for all $t \in \mbb{N}$, $\hst{h}{t} \in \hstspace{t}$, and $\at{t} \in \aspace$, the mapping $\prup{\cdot}{P_1}: x \in X \mapsto \prup{x}{P_1}\l(\hst{h}{t}, \at{t} \r) \in \mbb{R}$ is continuous. 
	\end{enumerate}
\end{assumption}
\begin{rem}
Assumptions~\ref{assmp:macpomdp:structured_X}(\textit{a}, \textit{b}) together imply uniform continuity of each mapping in the set,
\begin{align*}
&\big\{ \prup{\cdot}{P_1}: x \in X \mapsto \prup{x}{P_1}\l(\hst{h}{t}, \at{t} \r) \in \mbb{R} :  t \in \mbb{N}, \hst{h}{t} \in \hstspace{t}, \at{t} \in \aspace \big\}.
\end{align*}
\end{rem}
\noindent\textit{Roadmap of technical results:} We will show that the interaction mechanisms to be studied satisfy Assumptions~\ref{assmp:macpomdp:structured_X}(\textit{a,b}). Thereafter, we can compare them via intermediate results: Lemma~\ref{lem:macpomdp:continuityofmeasures_1} is used in Section~\ref{sec:dec_pps} to show the equivalence of joint mixtures of decentralized behavioral policy-profiles and joint mixtures of decentralized pure policy-profiles; and Lemma~\ref{lem:macpomdp:continuityofmeasures_2} is used in Section~\ref{sec:ciapproach} to prove the equivalence of behavioral coordination policies and joint mixtures of decentralized policy-profiles. The (omitted) proofs of the two lemmas follow easily using Assumptions~\ref{assmp:macpomdp:structured_X}(\textit{a,b}).
\begin{lem}\label{lem:macpomdp:continuityofmeasures_1}
	Let $X$ be a set of (valid) causal interaction-mechanisms which satisfies Assumptions~\ref{assmp:macpomdp:structured_X}(a,b). Let $z$ be a (valid) causal interaction-mechanism and suppose for every $\eps>0$, there exists $x \in X$ such that
	\begin{align*}
    \begin{split}
		&\l| \prup{x}{P_1} \l( \hst{h}{t}, \at{t} \r) - \prup{z}{P_1} \l( \hst{h}{t}, \at{t} \r) \r| \le \eps, \forall t \in \mbb{N}, \hst{h}{t} \in \hstspace{t}, \at{t} \in \aspace.
    \end{split}
	\end{align*}
	Then, there exists $x^\star = x^\star(z) \in X$ such that
	\begin{align*}
    \begin{split}
		& \prup{x^\star}{P_1} \l( \hst{h}{t}, \at{t} \r) 
		= \prup{z}{P_1} \l( \hst{h}{t}, \at{t} \r), 
       \forall t \in \mbb{N}, \hst{h}{t} \in \hstspace{t}, \at{t} \in \aspace.
    \end{split}
	\end{align*} 
\end{lem}
\begin{proof}
	See \cite[Lemma 2.2]{nkhanthesis2025}.\qedhere
\end{proof}

\begin{lem}\label{lem:macpomdp:continuityofmeasures_2}
	Let $X$ be a set of (valid) causal interaction-mechanisms which satisfies Assumptions~\ref{assmp:macpomdp:structured_X}(a,b). Suppose for some (valid) causal interaction-mechanism $z$, there exists a sequence $\{x_i\}_{i=1}^{\infty} \subseteq X$ such that each $x_i$ satisfies:
	\begin{align*}
    \begin{split}
		& \prup{x_i}{P_1} \l( \hst{h}{t}, \at{t} \r) = \prup{z}{P_1} \l( \hst{h}{t}, \at{t} \r), 
        \forall t \in [1, i]_{\mbb{Z}}, 
        \hst{h}{t} \in \hstspace{t}, \at{t} \in \aspace.
    \end{split}
	\end{align*}
	Then, there exists $x^\star = x^\star(z) \in X$ such that 
	\begin{align*}
    \begin{split}
		& \prup{x^\star}{P_1} \l( \hst{h}{t}, \at{t} \r) = \prup{z}{P_1} \l( \hst{h}{t}, \at{t} \r), 
		\forall t \in \mbb{N}, \hst{h}{t} \in \hstspace{t}, \at{t} \in \aspace.
    \end{split}
	\end{align*}
\end{lem}
\begin{proof}
	See \cite[Lemma 2.3]{nkhanthesis2025}.\qedhere
\end{proof}

\subsection{Occupation Measures}\label{sec:macpomdp:ocm}
For a (valid) causal interaction-mechanism $x$, causality leads to the following \emph{strategic independence} property: for all $t \in \mbb{N}$ and $B \in \borel{S}$,
\begin{align*}
	\prup{x}{P_1} \l( \Stt{t} \in B \big| \hst{h}{t}, \at{t} \r) = \pr_{P_1} \l( \Stt{t} \in B \big| \hst{h}{t}, \at{t}   \r). 
	\numberthis\label{eq:macpomdp:conditional_independence}
\end{align*}
This ensures that the corresponding conditional expectations of the immediate costs are also independent of $x$, i.e., for all $t\in \mbb{N}$,
\begin{align}
	\begin{split}
		\E{x}{P_1} \l[ \cCost \big| \hst{h}{t}, \at{t} \r] 
		& = 
		\cb \l(t, \hst{h}{t}, \at{t} \r), \text{ and } \\
		\E{x}{P_1} \l[ \dtCost{k} \big| \hst{h}{t}, \at{t} \r] 
		& =  
		\dbt{k} \l(t, \hst{h}{t}, \at{t} \r), k \in [K].
	\end{split}
	\label{eq:macpomdp:cthtat_and_dthtat}
\end{align}
Using \eqref{eq:macpomdp:cthtat_and_dthtat}, we obtain an inner-product representation for the long-term costs.
\begin{lem}
\label{lem:macpomdp:innerproduct_rep}\leavevmode
	Let $x$ be a (valid) causal interaction-mechanism. 
    Under Assumption~\ref{assmp:macpomdp:boundedcosts}(a), for every $(T, \alpha)$, 
    the following hold:
	\begin{align}
		\begin{split}
		\Ctn{T}{\alpha}(x) =  \ftn{T}{c}\l( \ocmtn{T}{\alpha}(x) \r) \defeq \dotp{\ocmtn{T}{\alpha}(x) }{ \boldcbt{T} }, \text{ and } \\
		\Dtn{k,T}{\alpha}(x) \!=\! \ftn{T}{d_k}\l( \ocmtn{T}{\alpha}(x) \r) \! \defeq \! \dotp{\ocmtn{T}{\alpha}(x) }{\bolddbt{k, T}}, k \in \! [K],
		\end{split}
		\label{eq:macpomdp:innerproduct_rep}
	\end{align}
	where,
	\begin{align*}
		\begin{split}
			& \ocmtn{T}{\alpha} (x) \! = \! \l\{ \alpha^{t-1} \prup{x}{P_1} \l( \hst{h}{t}, \at{t} \r) \!: t \in \! [1,T]_{\mbb{Z}}, \hst{h}{t} \in \hstspace{t}, \at{t} \in \aspace \r\} \!, \\
			& \boldcbt{T} \! = \! \l\{ \cb \l( t, \hst{h}{t}, \at{t} \r) \! : t \in \! [1,T]_{\mbb{Z}}, \hst{h}{t} \in \hstspace{t}, \at{t} \in \aspace \r\},\\
			& \bolddbt{k, T} \! = \! \l\{ \dbt{k} \l(t, \hst{h}{t}, \at{t} \r) \! : \! t \in \! [1,T]_{\mbb{Z}}, \hst{h}{t} \in \hstspace{t}, \at{t} \in \aspace \r\} \!, k \in \! [K].
		\end{split}
	\end{align*}
\end{lem}
\begin{proof} The proof is elementary except when $T=\infty$, which uses the Monotone Convergence Theorem. See \cite[Lemma 2.4]{nkhanthesis2025}.\qedhere
\end{proof}
In line with the prevailing usage in the \sacmdp \space literature, 
we refer to $\ocmtn{T}{\alpha}(x)$ as the \emph{(discounted) occupation measure}\footnote{Note that unlike the state, a joint-history $\hst{h}{t}$ can be visited only once (at its corresponding time $t$).} generated by $x$ since
\begin{align*}
\begin{split}
	& \sum_{t=1}^{T} \sum_{\hst{h}{t} \in \hstspace{t}} \sum_{\at{t} \in \aspace} \alpha^{t-1} \prup{x}{P_1} \l( \hst{h}{t}, \at{t} \r) = \sum_{t=1}^{T} \alpha^{t-1} \defeqr m(T, \alpha).
\end{split}
\end{align*}
Additionally, due to \eqref{eq:macpomdp:innerproduct_rep}, \eqref{eq:macpomdp:cop_1} has the following equivalent constrained optimization problem:
\begin{align}
	\begin{split}
		&\min \ftn{T}{c}(q) \notag\\
		&\text{s.t. } q \in \ocmtn{T}{\alpha} \l( X \r) \text{ and } \ftn{T}{d_k}(q) \le \constraintv_k \ \forall k \in [K],
	\end{split}
	\Bigg\}
	\tag{$ \coptn{T}{\alpha}\l(X\r)$-2}
	\label{eq:macpomdp:cop_2}
\end{align}
where $\ocmtn{T}{\alpha}(X) \subseteq \Mta{m(T, \alpha)}{\bigcup_{t=1}^{T} \l( \hstspace{t} \times \aspace \r) } $ is the set of all occupation measures that are generated through some interaction-mechanism in $X$. 

\begin{rem}\label{rem:macpomdp:ocm}
	A few remarks are in order:
	\begin{enumerate}
		\item \label{rem:macpomdp:ocm:pinf1} Assumption~\ref{assmp:macpomdp:structured_X}(\textit{b}) can be replaced by point-wise continuity of each component of the mapping,
		\begin{align*}
			\ocmtn{\infty}{1} \l( \cdot \r): x \in X \mapsto \ocmtn{\infty}{1} \l( x \r) \in \Mta{\infty}{\bigcup_{t=1}^{\infty} \l( \hstspace{t} \times \aspace \r) }.
		\end{align*}
		When comparing two interaction-mechanisms or sets thereof, the mapping $\ocmtn{\infty}{1}(\cdot)$ plays an important role---by comparing undiscounted probabilities over joint-history and joint-action pairs, one can present results for all permissible values of $(T, \alpha)$.
		
		\item If $\ocmtn{\infty}{1}\l( X \r)$ is convex, so is $\ocmtn{T}{\alpha} \l( X \r)$ for any $(T, \alpha)$.

		\item If $\ocmtn{\infty}{1}\l( X\r) \supseteq \ocmtn{\infty}{1}\l( X' \r)$, then $X$ \emph{dominates} $X'$---that is, any tuple of long-term objective and constraint costs 
        attained by a mechanism in $X'$ can also be attained by a mechanism in $X$. A classical example is in the single-agent setting, where behavioral policies are known to dominate pure policies.
	\end{enumerate} 
\end{rem}

We now present an important result for a generic set of interaction-mechanisms that satisfies Assumptions~\ref{assmp:macpomdp:structured_X}(\textit{a,b}).
\begin{thm}\label{thm:macpomdp:duality_W}
	Let $X$ denote a set of (valid) causal interaction-mechanisms which satisfies Assumptions~\ref{assmp:macpomdp:structured_X}(a,b). Then, under Assumption~\ref{assmp:macpomdp:boundedcosts}(a), the following statements hold for all $(T, \alpha)$:
	\begin{enumerate}
		\item[(A)]\label{thm:macpomdp:duality_W:continuity} The long-term costs, $\Ctn{T}{\alpha}$ and $\Dtn{k,T}{\alpha}$, $k \in [K]$, are lower semi-continuous. 
        If  Assumption~\ref{assmp:macpomdp:boundedcosts}(b) also holds, then they are uniformly continuous, 
		
		\item[(B)]\label{thm:macpomdp:duality_W:compact_ocm} The set $\ocmtn{T}{\alpha}(X)$, as a subspace of $\Mta{m(T, \alpha)}{\bigcup_{t=1}^{T} \l( \hstspace{t} \times \aspace \r)} $, is compact.
		
		\item[(C)]\label{thm:macpomdp:duality_W:continuity_dotp} The inner-product forms, 
        $\ftn{T}{c}$ and $\ftn{T}{d_k}$, $k \in [K]$, are lower semi-continuous. 
        If Assumption~\ref{assmp:macpomdp:boundedcosts}(b) also holds, then they are uniformly continuous. 

	\end{enumerate}
\end{thm}
\begin{proof}
The proof uses properties of compact metric spaces, Fatou's Lemma, and countability of the discrete space $\bigcup_{t=1}^{T} \l( \hstspace{t} \times \aspace \r)$. See \cite[Theorem 2.1]{nkhanthesis2025}. \qedhere
\end{proof}

%% file: Sections/decentralizedCase.tex
\renewcommand{\mbb}{\mathbb}
\renewcommand{\mcl}{\mathcal}

\section{Decentralized Policy-Profiles and their Mixtures}\label{sec:dec_pps}
 \subsection{System Model}\label{sec:dec_pps:model}
To formalize decentralized policy-profiles, we first define \emph{common} and \emph{private} histories:  
\textit{i}) The \emph{common history} of the team at time $t$, denoted by $\Hstn{t}{0} \in \hstnspace{t}{0}$, is defined as all the common observations the agents have seen up to time $t$, i.e., 
$\Hstn{t}{0} \defeq \Otn{1:t}{0}$.  
%
\textit{ii}) The \emph{private history} of agent $n$ at time $t$, denoted by $\Hstn{t}{n} \in \hstnspace{t}{n}$, is defined as all her private observations seen thus far and all her previous actions (excluding those that are part of the common observations), i.e., 
 \begin{align*}
 	\begin{split}
 		\Hstn{1}{n} &\defeq \Otn{1}{n}, \text{ and}\\
 		\Hstn{t}{n} &\defeq \l( \Hstn{t-1}{n}, \Atn{t-1}{n}, \Otn{t}{n} \r) \setminus \Hstn{t}{0} \ \forall t \in [2,\infty]_{\mbb{Z}}.
 	\end{split}
 \end{align*}
For simplicity, we assume that $\{\Hstn{t}{0},\Hstn{t}{1},\dotsc,\Hstn{t}{N}\}$ forms a partition of the joint-history $\Hst{t}$ at all times. So, $\Hst{t} = \Hstn{t}{0:N}$ for all $t$. 
\begin{rem}
To simplify notation, from now on, we assume that agents' actions are \udl{forever} private, i.e., $\Atn{1:t}{n}$ has nothing in common with $\Hstn{t}{0}$. So, $\hstnspace{t}{n} = \otnspace{1:t}{n} \times \atnspace{1:t-1}{n}$.
\end{rem}

\begin{dfn}[Decentralized Policy-Profiles]\label{dfn:dec_pps:beh_policy_profile}
	A \emph{decentralized behavioral policy-profile} is defined as a tuple $u=\un{1:N} \in \uspace \defeq \unspace{1:N} $ where $\un{n}$ denotes the behavioral policy used by agent $n$. Specifically, $\un{n}$ is a tuple of the form $ \utn{1:\infty}{n} \in \utnspace{1:\infty}{n}$ where $\utn{t}{n} \in \utnspace{t}{n}$ is the \emph{decision-rule} of agent $n$ for time $t$---a (Borel measurable) mapping from  $\hstnspace{t}{0, n} $ to $\Mta{1}{\anspace{n}}$ so that $\Atn{t}{n} \sim \utn{t}{n} ( \Hstn{t}{0, n} )$. 
	
	Decentralized pure policy-profiles (in which all agents choose their actions deterministically) lie inside $\uspace$. We denote the set of (Borel measurable) pure decision-rules of agent $n$ for time $t$ by $\utnspacepure{t}{n} \subseteq \utnspace{t}{n} $, the set of pure policies of agent $n$ by $\unspacepure{n} = \utnspacepure{1:\infty}{n} \subseteq \unspace{n}$, and the set of \emph{decentralized pure policy-profiles} by $\uspacepure = \unspacepure{1:N} \subseteq \uspace $.
\end{dfn}
We note that, due to decentralized randomizations, at any time $t$, given a joint-history $\hst{h}{t} \in \hstspace{t}$, the probability that the team forms joint-action $\at{t} \in \aspace$ is given by:
\begin{align*}
	\ut{t}\l(\at{t} \big|\hst{h}{t}\r) 
	\! \defeq \! \! \prod_{n=1}^{N} \! \utn{t}{n} \! \l( \! \hstn{h}{t}{0,n} \! \r) \! \l( \! \atn{t}{n}  \! \r) \!= \! \! \prod_{n=1}^{N} \! \utn{t}{n} \! \l( \! \atn{t}{n} \Big| \hstn{h}{t}{0, n} \! \r) \!.\numberthis\label{eq:dec_pps:uah}
\end{align*}
Causality is clear from \eqref{eq:dec_pps:uah}. Then, from the Ionesca-Tulcea theorem, it follows that there exists a unique probability measure $\prup{u}{P_1} = \prup{u}{P_1, \mcl{P}_{tr}}$ on the set of all infinite trajectories, namely $(\sspace \times \ospace \times \aspace)^{\infty} $, endowed with the Borel $\sigma$-algebra.\footnote{$\prup{u}{P_1}$ is referred to as a strategic measure in \cite{feinberg1996,piunovski1998}.} Thus, $\uspace$ is a set of (valid) causal interaction-mechanisms. 

Next, we ensure that $\uspace$ can form a compact metric space (Assumption~\ref{assmp:macpomdp:structured_X}(\textit{a})). Let 
\begin{align*}
	\begin{split}
		\hsnspace{0, n} \defeq \bigcup_{t= 1}^{\infty} \hstnspace{t}{0,n}, \text{ and } \quad 
		\hsspace \defeq \bigcup_{t=1}^{\infty} \hstspace{t},
	\end{split}
\end{align*}
and note that $\uspace$ and each $\unspace{n}$, $n \in [N]$, are, respectively, in one-to-one correspondence with the sets,
\begin{align*}
	\xtspace{\uspace} \! \defeq \! \prod_{n=1}^{N} \! \xtspace{\unspace{n}}, \text{ and } 
	\xtspace{\unspace{n}} \! \defeq \hspace{-20pt} \prod_{\hsn{h}{0,n} \in \hsnspace{0,n}} \Mta{1}{\anspace{n}; \hsn{h}{0,n}}.
\end{align*}
Since each $\Mta{1}{\anspace{n}; \hsn{h}{0,n}}$ is a compact metric space\footnote{This follows from Prokhorov's theorem. Being a finite discrete space, $\anspace{n}$ is a compact metric space, so complete and separable.}, endowing $\xtspace{\uspace}$ and $\xtspace{\unspace{n}}$, $n \in [N]$ with the product topologies, we ensure that $\uspace$ and $\unspace{n}$, $n\in [N]$, all form compact metric spaces.\footnote{This holds via Tychonoff's theorem and the fact that countable product of metric spaces is metrizable under the product topology.} Then, additionally, the sets $\unspacepure{n}$ and $\uspacepure$, being closed in $\unspace{n}$ and $\uspace$ respectively, also form compact metric spaces.
\begin{prop}\label{prop:dec_pps:subspace_topologies}
	The subspace topologies of $\unspacepure{n} \subseteq \unspace{n}$, $n \in [N]$, and $\uspacepure \subseteq \uspace$ are compact and metrizable.
\end{prop}

In keeping with \sacmdp s,  
the first choice for the team's search space would be $\uspace$. However, because agents randomize independently to generate the joint-action---see \eqref{eq:dec_pps:uah}---, the resulting set of occupation measures, namely $\ocmtn{T}{\alpha}(\uspace)$, will, in general, be non-convex and thus (strong) Lagrangian duality may not hold---see \cite[Example 5.1]{altman-constrainedMDP}, \cite[Example 1]{chen2024hardness}, and \cite[Proposition 3.3]{nkhanthesis2025}. Hence, we will enlarge the search space by introducing an additional layer of randomization. Since $\uspace$ has been given a topology, we can use the set of probability measures 
on its Borel $\sigma$-algebra, namely $\Mta{1}{\uspace}$, which by Prokhorov's theorem, is a compact metric space under the topology of weak convergence. We denote an element of $\Mta{1}{\uspace}$ by $\mu$ and refer to it as a \emph{joint mixture of decentralized behavioral policy-profiles}. Under $\mu \in \Mta{1}{\uspace}$, the behavioral policy-profile of the team becomes a random object $U$, and the system trajectory 
includes an additional time-step 0 where $U$ is realized. Let $\prup{\mu}{P_1} = \prup{\mu}{P_1, \mcl{P}_{tr}}$ be the unique probability measure on the set of all infinite trajectories, namely $\uspace \times (\sspace \times \ospace \times \aspace)^{\infty} $, endowed with the Borel $\sigma$-algebra, and let $\E{\mu}{P_1} = \E{\mu}{P_1, \mcl{P}_{tr}}$ denote the corresponding expectation operator. Then, we have:
\begin{align*}
\begin{split}
	&\Ctn{T}{\alpha} \! \l(\mu\r)
	\!=\! \E{\mu}{P_1} \! \l[ \! \Ctn{T}{\alpha} \! \l(U\r) \! \r]   
	 \text{, and } 
	\Dtn{T}{\alpha} \! \l(\mu\r) \!=\! \E{\mu}{P_1} \! \l[ \! \Dtn{T}{\alpha} \! \l(U\r) \r].
\end{split}
\end{align*}
Sampling a behavioral policy-profile using $\mu \in \Mta{1}{\uspace}$ necessitates joint randomization (hence, the qualifier joint) or some form of communication between agents.\footnote{For example, one of the agents could use the measure $\mu \in \Mta{1}{\uspace} $ to draw $U \in \uspace$, and then communicate it to the rest of the team.} 
There are a few subsets of $\Mta{1}{\uspace}$ that we will discuss:
\begin{enumerate}
\item Joint mixtures of decentralized pure policy-profiles, $\Mta{1}{\uspacepure}$; 
\item Product mixtures of decentralized behavioral policy-profiles, $\prod_{n=1}^{N} \Mta{1}{\unspace{n}} $; and 
\item Product mixtures of decentralized pure policy-profiles, $\prod_{n=1}^{N} \Mta{1}{\unspacepure{n}} $.
\end{enumerate}
Like $\uspace$ and $\Mta{1}{\uspace}$, all these subsets satisfy Assumption~\ref{assmp:macpomdp:structured_X}(\textit{a}). 
\begin{prop}\label{prop:dec_pps:subspace_topologies_2}
The subspace topology of $\Mta{1}{\uspacepure} \subseteq \Mta{1}{\uspace}$ and the product topologies of $\prod_{n=1}^{N} \Mta{1}{\unspace{n}} $, and $\prod_{n=1}^{N} \Mta{1}{\unspacepure{n}} $ are compact and metrizable.
\end{prop}
\begin{proof}
	See \cite[Proposition 3.2]{nkhanthesis2025}.
\end{proof}

\subsection{Analytic and Structural Results}\label{sec:dec_pps:duality}
With $\uspace$, $\Mta{1}{\uspace}$, $\Mta{1}{\uspacepure} $, $\prod_{n=1}^{N} \Mta{1}{\unspace{n}} $, and $\prod_{n=1}^{N} \Mta{1}{\unspacepure{n}}$ all forming compact metric spaces, we now investigate the compactness and convexity properties of their achievable occupation measures and the relationships between them.



\begin{thm}\label{thm:dec_pps:continuity_ocm}
	Let $X$ denote one of the spaces,  $\uspace$, $\Mta{1}{\uspace}$, $\Mta{1}{\uspacepure} $, $\prod_{n=1}^{N} \Mta{1}{\unspace{n}} $, $\prod_{n=1}^{N} \Mta{1}{\unspacepure{n}}$. Then, each component of the mapping, 
	$$\ocmtn{\infty}{1} (\cdot) : x \in X \mapsto \ocmtn{\infty}{1}(x) \in \Mta{\infty}{\bigcup_{t=1}^{\infty} \l( \hstspace{t} \times \aspace \r) },$$
	is pointwise continuous. Thus, by Theorem~\ref{thm:macpomdp:duality_W}(\textit{B}), for any $(T, \alpha)$, $\ocmtn{T}{\alpha} \l( X \r) $ is a compact subspace of $\Mta{m(T,\alpha)}{\bigcup_{t=1}^{T} \l( \hstspace{t}\times\aspace \r) }$.
\end{thm}
\begin{proof}
	The proof follows directly from Lemmas~\ref{lem:dec_pps:puth} and \ref{lem:dec_pps:pmuth}, and Corollary~\ref{cor:dec_pps:pw_continuity}. \qedhere
\end{proof}
\begin{lem}\label{lem:dec_pps:puth}
	Let $\l\{ \qseq{u}{i} \r\}_{i=1}^{\infty} \subseteq \uspace$ be a sequence that converges to $u \in \uspace$. Then, for any $t \in \mbb{N}$, $ \hst{h}{t} \in \hstspace{t} $, and $\at{t} \in \mcl{A}$,
	 \begin{align*}
		 	\lim_{i\ra \infty} \prup{\qseq{u}{i}}{P_1} \l(\hst{h}{t}, \at{t} \r) &= \prup{u}{P_1} \l(\hst{h}{t}, \at{t} \r).
		 \end{align*}
\end{lem}
\begin{proof}
Uses induction. See \cite[Lemma 3.1]{nkhanthesis2025}. \qedhere
\end{proof}

\begin{lem}
	\label{lem:dec_pps:pmuth}
	Let $\l\{ \mu_i \r\}_{i=1}^{\infty} \subseteq \Mta{1}{\uspace}$ be a sequence that converges to $\mu \in \Mta{1}{\uspace}$. Then, for any $t \in \mbb{N}$, $ \hst{h}{t} \in \hstspace{t} $, and $\at{t} \in \mcl{A}$,
	 \begin{align*}
	 	\lim_{i\ra \infty} \prup{\mu_i}{P_1} \l( \hst{h}{t}, \at{t} \r) 
	 	&= \prup{\mu}{P_1} \l( \hst{h}{t}, \at{t} \r).
	 \end{align*}
\end{lem}
\begin{proof}
By noting the relation,
\begin{align*}
\begin{split}
& \prup{\mu_i}{P_1} \l( \hst{h}{t}, \at{t} \r) =
P_1\l( \sspace, \hst{h}{1} \r) \prod_{t'=2}^{t} \pr_{P_1} \l( \ot{t'} \big| \hst{h}{t'-1}, \at{t'-1} \r)\\
	& \qquad	\times \int_{\uspace}
		\prod_{n=1}^{N} \prod_{t'=1}^{t} \utn{t'}{n}\l( \atn{t'}{n} \big|  \hstn{h}{t'}{0, n} \r) \mu_i\l( du \r),
\end{split}\numberthis\label{eq:dec_pps:pmuhtat}
\end{align*}
the proof uses Lemma~\ref{lem:dec_pps:puth} in an induction argument. See \cite[Lemma 3.2]{nkhanthesis2025}.\qedhere
\end{proof}

\begin{cor}
	\label{cor:dec_pps:pw_continuity}
	Let $X$ be one of the sets $\Mta{1}{\uspacepure}$, $\prod_{n=1}^{N} \Mta{1}{\unspace{n}} $, and $ \prod_{n=1}^{N} \Mta{1}{\unspacepure{n}} $, and let $\l\{ x_i \r\}_{i=1}^{\infty} \subseteq X $ be a sequence that converges to $x \in X$. Then, for any $t \in \mbb{N}$, $ \hst{h}{t} \in \hstspace{t} $, and $\at{t} \in \mcl{A}$,
	 \begin{align*}
	 	\lim_{i\ra \infty} \prup{x_i}{P_1} \l( \hst{h}{t}, \at{t} \r) 
	 	&= \prup{x}{P_1} \l( \hst{h}{t}, \at{t} \r).
	 \end{align*}
\end{cor}
 \begin{proof} See \cite[Corollary 3.1]{nkhanthesis2025}.\qedhere
 \end{proof}

The next result establishes convexity of the occupation measures realized from $\Mta{1}{\uspace}$ and $\Mta{1}{\uspacepure}$.
\begin{cor}\label{cor:dec_pps:convexity}
The sets  $\ocmtn{\infty}{1} \l( \Mta{1}{\uspace} \r)$ and $\ocmtn{\infty}{1} \l( \Mta{1}{\uspacepure} \r)$ are convex. 
\end{cor}
\begin{proof}
Follows from \eqref{eq:dec_pps:pmuhtat} and the convexity of $\Mta{1}{\uspace}$ and $\Mta{1}{\uspacepure}$.\qedhere
\end{proof}

Finally, we establish two key equivalence relations (in terms of achievable occupation measures) within decentralized policy-profiles and their mixtures. These relations are stated in Theorems~\ref{thm:dec_pps:ocm} and \ref{thm:dec_pps:ocm_3}.

\begin{thm}\label{thm:dec_pps:ocm}
We have:
\begin{align*}
	\ocmtn{\infty}{1} \! \l( \uspace \r) \!=\! \ocmtn{\infty}{1} \! \l( \prod_{n=1}^{N} \! \Mta{1}{\unspace{n}} \! \r) \! = \! \ocmtn{\infty}{1} \! \l( \prod_{n=1}^{N} \! \Mta{1}{\unspacepure{n}} \! \r) \!. 
\end{align*}
\end{thm}
\begin{proof}
	The proof follows from Lemmas~\ref{lem:dec_pps:ocm_1} and \ref{lem:dec_pps:ocm_2b}.\qedhere
\end{proof}
\begin{lem}\label{lem:dec_pps:ocm_1}
	We have:
    \begin{align*}
        &\ocmtn{\infty}{1}\l( \uspace \r) \supseteq  \ocmtn{\infty}{1}\l(  \prod_{n=1}^{N} \Mta{1}{\unspace{n}}  \r)  \\
        &\hspace{100pt} \supseteq  \ocmtn{\infty}{1} \l( \! \prod_{n=1}^{N} \Mta{1}{\unspacepure{n}}  \r) \!.
    \end{align*}
\end{lem}
\begin{proof}
	The second inclusion is trivial. To see the first inclusion, fix a product mixture of decentralized behavioral policy-profiles $\udl{\mu} \in \prod_{n=1}^{N} \Mta{1}{\unspace{n}} $. We want to show that there exists a decentralized behavioral policy-profile $\udl{u}=\udl{u}(\udl{\mu}) \in \uspace$, such that 
	\begin{align*}
		\ocmtn{\infty}{1}\l( \udl{\mu} \r) = \ocmtn{\infty}{1} \l( \udl{u} \r).
	\end{align*}
	To that end, define $\udl{u}$ such that
	\begin{align*}
    \begin{split}
		&\qt{\udl{u}}{t} \l( \at{t} | \hst{h}{t}\ \r) = \prod_{n=1}^{N} 
		\qtn{\udl{u}}{t}{n} 
		\l( \atn{t}{n} | \hstn{h}{t}{0, n} \r) \\ 
		& = \begin{cases}
			\frac{ \prup{\udl{\mu}}{P_1} \l( \hst{h}{t}, \at{t} \r)
			}{\prup{\udl{\mu}}{P_1} \l( \hst{h}{t} \r)} &\text{if } \prup{\udl{\mu}}{P_1} \l( \hst{h}{t} \r) \ne 0;\\
			\prod_{n=1}^{N} \frac{1}{|\anspace{n}|} &\text{otherwise}.
		\end{cases}
    \end{split}\numberthis\label{eq:dec_pps:dominance:u}
	\end{align*}
	The above assignment is correct because the right hand side in \eqref{eq:dec_pps:dominance:u} is a fully factorized function of $\atn{t}{n}$'s. To see this, note that $\udl{\mu}$ is a product measure. Thus, using Tonelli's Theorem 
    in \eqref{eq:dec_pps:pmuhtat} gives
	\begin{align*}
		\prup{\udl{\mu}}{P_1} \l( \hst{h}{t}, \at{t} \r) 
		&=P_1\l( \sspace, \hst{h}{1} \r) 
		\prod_{t'=2}^{t} \pr_{P_1} \l( \ot{t'} | \hst{h}{t'-1}, \at{t'-1} \r) \\
		& \hspace{-30pt} \times \prod_{n=1}^{N} \int_{\unspace{n}}  \prod_{t'=1}^{t} \utn{t'}{n}\l( \atn{t'}{n} |  \hstn{h}{t'}{0, n} \r) \qn{\udl{\mu}}{n} \l( d\un{n}\r),\\
		\prup{\udl{\mu}}{P_1} \l( \hst{h}{t} \r)
		&=P_1\l( \sspace, \hst{h}{1} \r) 
		\prod_{t'=2}^{t} \pr_{P_1} \l( \ot{t'} | \hst{h}{t'-1}, \at{t'-1} \r) \\
		&\hspace{-30pt} \times \prod_{n=1}^{N} \int_{\unspace{n}}  \prod_{t'=1}^{t-1} \utn{t'}{n}\l( \atn{t'}{n} |  \hstn{h}{t'}{0, n} \r) \qn{\udl{\mu}}{n}\l( d\un{n}\r).
	\end{align*}
	Thus, 
	\begin{align*}
		&\frac{ \prup{\udl{\mu}}{P_1} \l(\hst{h}{t}, \at{t} \r)
		}{\prup{\udl{\mu}}{P_1} \l( \hst{h}{t} \r) } \\
		&\hspace{5pt} =
		\frac{ \prod_{n=1}^{N} \int_{\unspace{n}} \prod_{t'=1}^{t} \utn{t'}{n}\l( \atn{t'}{n} |  \hstn{h}{t'}{0, n} \r) \qn{\udl{\mu}}{n} \l( d\un{n}\r) }{ \prod_{n=1}^{N} \int_{\unspace{n}} \prod_{t'=1}^{t-1} \utn{t'}{n}\l( \atn{t'}{n} |  \hstn{h}{t'}{0, n} \r) \qn{\udl{\mu}}{n} \l( d \un{n}\r) },
	\end{align*}
	where both the numerator and denominator are products of some functions of the form $f^{(n)}\l( \atn{t}{n}, \hstn{h}{t}{0, n} \r) $. We now prove, by forward induction, that for all $t\in\mbb{N}$, $\udl{u}$ and $\udl{\mu}$ induce the same distribution on pairs of the form $\l( \Hst{t}, \At{t} \r)$.
		
		\smallskip\smallskip
		\noindent \udl{\textit{Base Case}}: For time $t=1$, let $\hst{h}{1} \in \hstspace{1} = \ospace$ and $\at{1} \in \aspace$. We have
		\begin{align*}
			\prup{\udl{\mu}}{P_1} \l(\hst{h}{1}, \at{1} \r) &= P_1\l( \sspace, \hst{h}{1} \r) \int_{\uspace} \udl{\mu} \l( du \r) \ut{1}\l( \at{1} | \hst{h}{1} \r),
		\end{align*}
		and
		\begin{align*}
			\prup{\udl{u}}{P_1} \l( \hst{h}{1}, \at{1} \r)
			&= P_1\l( \sspace, \hst{h}{1} \r) \qt{\udl{u}}{1}\l( \at{1} | \hst{h}{1} \r) \\
			&= P_1\l( \sspace, \hst{h}{1} \r) 
			\frac{\prup{\udl{\mu}}{P_1} \l( \hst{h}{1}, \at{1} \r)}{\prup{\udl{\mu}}{P_1} \l( \hst{h}{1} \r) }
			=\prup{\udl{\mu}}{P_1} \l( \hst{h}{1}, \at{1} \r),
		\end{align*}
		where the last step holds as $ \prup{\udl{\mu}}{P_1} \l( \hst{h}{1} \r) = P_1\l( \sspace, \hst{h}{1} \r) $.
		
		\smallskip\smallskip
		\noindent \udl{\textit{Induction Step}}. 
		Assuming that the statement is true for time $t$, we show that it is true for time $t+1$. Let $\hst{h}{t+1} = \l( \ot{1:t+1}, \at{1:t} \r) = \l( \hst{h}{t}, \at{t}, \ot{t+1} \r) \in \hstspace{t+1}$ and $\at{t+1} \in \aspace$. We have
		\begin{align*}
			\prup{\udl{\mu}}{P_1} \l( \hst{h}{t+1} \r) 
			&\hspace{0pt}= \prup{\udl{\mu}}{P_1} \l(\hst{h}{t}, \at{t} \r) \pr_{P_1} \l( \ot{t+1} \big| \hst{h}{t}, \at{t} \r)
						\\&\hspace{0pt}
			\labelrel{=}{eqr:dominance:ind} \prup{\udl{u}}{P_1} \l( \hst{h}{t}, \at{t} \r)
			\pr_{P_1} \l( \ot{t+1} | \hst{h}{t}, \at{t} \r)
						\\&\hspace{0pt}
			=\prup{\udl{u}}{P_1} \l( \hst{h}{t+1} \r),
		\end{align*}
		where, \eqref{eqr:dominance:ind} uses the inductive hypothesis. The above equality along with the definition of $\udl{u}$ in \eqref{eq:dec_pps:dominance:u} implies 
		\begin{align*}
			\prup{\udl{u}}{P_1} \l( \hst{h}{t+1}, \at{t+1} \r)
			\hspace{0pt} 
			&\!=\! \prup{\udl{u}}{P_1} \l( \hst{h}{t+1} \r) \qt{\udl{u}}{t+1} \l(\at{t+1} | \hst{h}{t+1} \r) 
						\\ \hspace{0pt}
			&\!=\! \prup{\udl{\mu}}{P_1} \l( \hst{h}{t+1} \r) \frac{\prup{\udl{\mu}}{P_1} \l( \hst{h}{t+1}, \at{t+1} \r) }{\prup{\udl{\mu}}{P_1} \l( \hst{h}{t+1} \r)}
						\\ \hspace{0pt} 
			&\!=\! \prup{\udl{\mu}}{P_1} \l( \hst{h}{t+1}, \at{t+1} \r).  
		\end{align*}
	This completes the proof.\qedhere
\end{proof}

\begin{lem}\label{lem:dec_pps:ocm_2b}
We have: 
$$ \ocmtn{\infty}{1} \l( \prod_{n=1}^{N} \Mta{1}{\unspacepure{n}} \r) \supseteq \ocmtn{\infty}{1} \l( \uspace \r).$$
\end{lem}
\begin{proof}
	Fix a decentralized behavioral policy-profile $\udl{u} \in \uspace $. For each $n \in [N]$, let $$\qn{e}{n}: \mbb{N} \ra \bigcup_{t=1}^{\infty} \{t\} \times \hstnspace{t}{0, n}$$ 
	be a bijective mapping and consider a collection of mutually independent random variables $\l\{ \qtn{Y}{i}{n}: i \in \mbb{N} \r\}$ such that each $\qtn{Y}{i}{n}$ takes value $\delta_{\an{n}} $ with probability 
	$\qtn{\udl{u}}{t}{n} \l( \an{n} \big| \hstn{h}{t}{0, n} \r)  $ where $\l(t, \hstn{h}{t}{0, n} \r) = \qn{e}{n} \l( i \r)$. From the Kolmogorov extension theorem, it follows that there exists a unique probability measure $\qn{\udl{\mu}}{n} \in \Mta{1}{\unspacepure{n}} $ such that for all $t \in \mbb{N}$ and $\l(\otn{1:t}{0,n} \atn{1:t}{n}\r) \in \hstnspace{t}{0,n} \times \anspace{n}$, we have:
	\begin{align*}
		&\qn{\udl{\mu}}{n} \Big( \Big\{ \un{n} \in \unspacepure{n}: \utn{t'}{n} \Big( \otn{1:t'}{0,n}, \atn{1:t'-1}{n} \Big) =\delta_{\atn{t'}{n}} \\
       &\hspace{14em}  \text{ for all } t' \in [1, t]_{\mbb{Z}} \Big\} \Big) \\
		& = 
		\prod_{t'=1}^{t} \qtn{\udl{u}}{t'}{n} \l( \atn{t'}{n} | \otn{1:t'}{0,n}, \atn{1:t'-1}{n} \r).
	\end{align*}
	Let $\udl{\mu} = \otimes_{n=1}^N \qn{\udl{\mu}}{n}$. From \eqref{eq:dec_pps:pmuhtat}, it is then clear that $\ocmtn{\infty}{1}\l( \udl{u} \r) = \ocmtn{\infty}{1} \l( \udl{\mu} \r)$. 
	Since $\udl{u}$ was arbitrary, this concludes the proof. \qedhere
\end{proof}

\begin{thm}\label{thm:dec_pps:ocm_3}
We have:
\begin{align*}
	\ocmtn{\infty}{1} \l( \Mta{1}{\uspace} \r) = \ocmtn{\infty}{1} \l( \Mta{1}{\uspacepure} \r).
\end{align*}
\end{thm}
\begin{proof}
	See \cite[Theorem 3.4]{nkhanthesis2025}.
\end{proof}

%% file: Sections/coordinatedCase.tex
\section{Coordination Policies and their Mixtures}\label{sec:ciapproach}
For technical reasons, in this section, we make the following additional assumption. 
\begin{assumption}[Finite Number of Observations]\label{assmp:ciapproach:finite_observations}
	The joint-observation space $\ospace$ is finite.
\end{assumption}
\subsection{System Model}\label{sec:ciapproach:model}
We will use the coordinator's approach (the common information approach) \cite{nayyar13,nayyar14} to formulate \emph{coordination policies} and their mixtures. These mechanisms rely on a virtual agent, called the \emph{coordinator} who, at any time $t$, chooses a tuple of $N$ mappings, represented as:
\begin{align*}
	\Gammat{t} = \Gammatn{t}{1:N}, \Gammatn{t}{n}: \hstnspace{t}{n} \ra \anspace{n}, n \in [N].
\end{align*}
Here, each $\Gammatn{t}{n}$ acts as a \emph{prescription} for agent $n$---agent $n$ applies it to her private history $\Hstn{t}{n}$ to choose her action. Formally,
\begin{align*}
	\Atn{t}{n} = \Gammatn{t}{n} \l( \Hstn{t}{n} \r).
\end{align*}
Thus, the joint-action $\At{t}$ is given by:
\begin{align*}
	\At{t} = \l( \Gammatn{t}{n}\l(\Hstn{t}{n}\r): n \in [N] \r) \defeqr \Gammat{t}\l(\Hstn{t}{1:N}\r).
\end{align*}
The set of all prescriptions at time $t$ is given by:
\begin{align*}
	\gammatspace{t} \! = \! \gammatnspace{t}{1:N}, \gammatnspace{t}{n} \! \defeq \! \l\{ \gammatn{t}{n}: \hstnspace{t}{n} \ra \anspace{n} \r\}, n \in \! [N].
\end{align*}
Here, noting that $\gammatnspace{t}{n}$, $n \in [N]$ has a one-to-one correspondence with the set $\l(\anspace{n}\r)^{| \hstnspace{t}{n} |}$ which is a compact metric space under the product topology, one ensures that $\gammatspace{t}$ and $\gammatnspace{t}{n}$, $n \in [N]$, all form compact metric spaces. 

Now, we formalize how the coordinator generates her prescriptions. To select $\Gammat{t} \in \gammatspace{t}$, the coordinator uses the common observations $\Hstn{t}{0}$ and her previous prescriptions $\Gammat{1:t-1}$. The coordinator has perfect recall as she retains all past information. The aggregate information, referred to as the \emph{prescription-observation history} at time $t$, is denoted by $\wtHstn{t}{0}$, defined recursively as:
\begin{align*}
	\begin{split}
		\wtHstn{1}{0} &\defeq \Otn{1}{0}, \text{ and } \\
		\wtHstn{t}{0} &\defeq \l( \wtHstn{t-1}{0}, \Gammat{t-1}, \Otn{t}{0} \r) \forall t\in [2,\infty]_{\mbb{Z}}.
	\end{split}
\end{align*}
The set of all possible prescription-observation histories at time $t$ is given by $\wthstnspace{t}{0} = \hstnspace{t}{0} \times \gammatspace{1:t-1} $. When endowed with the product topology, this set forms a compact metric space, enabling the formal definition of coordination policies.
\begin{dfn}[Coordination Policies]\label{dfn:ciapproach:coord_ps}
	A \emph{behavioral coordination policy} is defined as a tuple $v = \vt{1:\infty} \in \vvspace \defeq \vtspace{1:\infty}$, where $\vt{t} \in \vtspace{t}$ is the decision-rule of the coordinator for time $t$. Each $\vt{t}$ is a Borel measurable mapping from $\wthstnspace{t}{0} $ to $\Mta{1}{\gammatspace{t}}$ so that $\Gammat{t} \sim  \vt{t} \l( \wtHstn{t}{0} \r)$. 
	
	Pure coordination policies (in which all prescriptions are generated deterministically) lie inside $\vvspace$. The set of (Borel measurable) pure decision-rules of the coordinator for time $t$ is denoted by $\vtspacepure{t}$, and the set of all pure coordination policies is denoted by $\vvspacepure  =  \vtspacepure{1:\infty} $.	
\end{dfn}
The system formed when agents follow a coordination policy is referred to as the \emph{coordinated system}. 
%
For any $v\in\vvspace$, let $\prup{v}{P_1} = \prup{v}{P_1, \mcl{P}_{tr}} $ be the unique probability measure\footnote{See the Ionesca-Tulcea theorem.} on the set of all infinite trajectories, namely $\prod_{t=1}^{\infty} \l(\sspace \times \ospace \times \gammatspace{t} \times \aspace\r) $, endowed with the Borel $\sigma$-algebra, and let $\E{v}{P_1} = \E{v}{P_1, \mcl{P}_{tr}}$ be the corresponding expectation operator. Since, by construction, all agents follow the coordinator's prescriptions, we have the following strategic independence property: for all $t\in \mbb{N}$, 
\begin{align*}
\begin{split}
	& \prup{v}{P_1} \l( \hstn{h}{t}{1:N}, \at{t} \big| \wthstn{t}{0}, \gammat{t} \r) = \\
    & \qquad \pr_{P_1} \l( \hstn{h}{t}{1:N} \big| \wthstn{t}{0}  \r)
     \indevent{\at{t} = \gammat{t} \l( \hstn{h}{t}{1:N} \r) } .
\end{split}
\end{align*}
This property and  \eqref{eq:macpomdp:conditional_independence}  
ensure that 
for all $t\in \mbb{N}$,
\begin{align*}
	\begin{split}
		 \E{v}{P_1} \l[ \cCost \big| \wthstn{t}{0}, \gammat{t} \r] 
		 & = 
        \cc \l( t, \wthstn{t}{0}, \gammat{t} \r) , \text{ and } \\
		\E{v}{P_1} \l[ \dtCost{k} \big| \wthstn{t}{0}, \gammat{t} \r] 
		& 
        = \dct{k} \l( t, \wthstn{t}{0}, \gammat{t} \r), k \in [K].
	\end{split}
\end{align*}

To ensure $\vvspace$ can form a compact metric space, we note its one-to-one correspondence with the set
\begin{align*}
	\mcl{X}_{\vvspace} = \prod_{ \wthsn{0} \in \bigcup_{t=1}^{\infty} \wthstnspace{t}{0}  } 
	\Mta{1}{\gammatspace{t}; \wthsn{0} }.\numberthis\label{eq:ciapproach:xvspace}
\end{align*}
Since each $\Mta{1}{\gammatspace{t}; \wthsn{0}}$ is a compact metric space under the topology of weak convergence\footnote{This follows from Prokhorov's theorem.}, by endowing $\xtspace{\vvspace}$ with the product topology, one ensures that $\vvspace$ can form a compact metric space under Assumption~\ref{assmp:ciapproach:finite_observations} (which ensures each $\wthstnspace{t}{0}$ is countable (actually finite)). With this choice, the set $\vvspacepure$ remains closed in $\vvspace$, and thus, is a compact metric space as well. Finally, we also introduce the set $\Mta{1}{\vvspace}$, which (again by Prokhorov's theorem) is a compact metric space under the topology of weak convergence. Also, because $\vvspacepure$ is closed in $\vvspace$, $\Mta{1}{\vvspacepure} \subseteq \Mta{1}{\vvspace} $ forms a compact metric space under the subspace topology. 

The space $\Mta{1}{\vvspace}$ contains \emph{mixtures of coordination policies} and we denote its typical element by $\nu$. The decision making interpretation of $\Mta{1}{\vvspace}$ is similar to that discussed for $\Mta{1}{\uspace}$ in Section~\ref{sec:dec_pps}---the system trajectory 
now includes time-step 0 where the (random) coordination policy, $V$, is realized. Let $\prup{\nu}{P_1} = \prup{\nu}{P_1, \mcl{P}_{tr}}$ denote the unique probability measure on the set of all infinite trajectories, namely $\vvspace \times \prod_{t=1}^{\infty} (\sspace \times \ospace \times \gammatspace{t} \times \aspace)$, endowed with Borel $\sigma$-algebra, and let $\E{\nu}{P_1} = \E{\nu}{P_1, \mcl{P}_{tr}}$ be the corresponding expectation operator. Then, we have:
\begin{align*}
	\Ctn{T}{\alpha}(\nu) = \E{\nu}{P_1} \l[ \Ctn{T}{\alpha}\l(V\r) \r],  
	\Dtn{T}{\alpha}(\nu) = \E{\nu}{P_1} \l[ \Dtn{T}{\alpha}\l(V\r) \r].
\end{align*}
We note that, with the exception of $\vvspacepure$, coordination policies and their mixtures, like joint mixtures of decentralized policy-profiles, require common randomness.

\subsection{Analytic and Structural Results}\label{sec:ciapproach:duality}
When the number of observations is finite, the set of prescriptions available to the coordinator at any time is also finite. Then, the coordinated system (a cooperative multi-agent system) can be viewed as a single agent system with finite action spaces, and the following proposition easily follows.

\begin{prop}\label{prop:ciapproach:ocm_properties}
Let $X$ denote one of the spaces, $
\vvspace $, $\Mta{1}{\vvspace}$, $\Mta{1}{\vvspacepure}$. Then, under Assumption~\ref{assmp:ciapproach:finite_observations}, the following statements hold:
\begin{enumerate}
%
	
	\item[(A)] Each component of the mapping,
	$$\ocmtn{\infty}{1} (\cdot): x \in X \mapsto \ocmtn{\infty}{1} \l( x \r),$$
	is pointwise continuous. Thus, by Theorem~\ref{thm:macpomdp:duality_W}(B), for all $(T, \alpha)$, $\ocmtn{T}{\alpha}\l( X\r)$ is a compact subspace of $\Mta{m(T, \alpha)}{\bigcup_{t=1}^{T} \l(\hstspace{t}\times\aspace \r)}$.
	
	\item[(B)] The set $ \ocmtn{\infty}{1} \l( X \r)$ is convex, and 
	\begin{align*}
		\ocmtn{\infty}{1} \l( X \r) = \ocmtn{\infty}{1} \l( \vvspace \r).
	\end{align*}
\end{enumerate}
\end{prop}
 \begin{proof}See \cite[Proposition 4.2]{nkhanthesis2025}.\qedhere 
		
		
		
 \end{proof}

Finally, letting $X$ denote $\vvspace$, $\Mta{1}{\vvspace}$ or $\Mta{1}{\vvspacepure}$ and letting $Y$ denote $\Mta{1}{\uspace}$ or $\Mta{1}{\uspacepure}$, 
we show\footnote{The proof of Theorem~\ref{thm:ciapproach:equivalence} also yields an alternative proof for the equivalence in \cite{nayyar13}.} that the set of occupation measures of $X$ and $Y$ are the same.
\begin{thm}\label{thm:ciapproach:equivalence}
	Under Assumption~\ref{assmp:ciapproach:finite_observations}, we have:
	\begin{align*}
		\ocmtn{\infty}{1} \l( \Mta{1}{\uspacepure} \r) = \ocmtn{\infty}{1} \l( \Mta{1}{\vvspacepure} \r).\numberthis\label{eq:ciapproach:ocms_m1upure_and_m1vpure}
	\end{align*}
\end{thm}
\begin{proof}\leavevmode
See \cite[Theorem 4.2]{nkhanthesis2025}.
\end{proof}


%% file: Sections/conclusion.tex
\section{Conclusion}\label{sec:conclusion}
In this paper, we considered a cooperative multi-agent constrained POMDP with a general state space, countable observations, and finite actions. We introduced decentralized policy profiles and their mixtures---namely, $\uspace$, $\uspacepure$, $\Mta{1}{\uspace}$, $\Mta{1}{\uspacepure}$, $\prod_{n=1}^{N} \Mta{1}{\unspace{n}}$, and $\prod_{n=1}^{N} \Mta{1}{\unspacepure{n}} $---and, coordination policies and their mixtures---namely $\vvspace$, $\vvspacepure$, $\Mta{1}{\vvspace}$, and $\Mta{1}{\vvspacepure}$---and, showed the following: 
\textit{i}) independently randomized decentralized policy-profiles (both pure and behavioral) have the same set of occupation measures as decentralized behavioral policy-profiles (Theorem~\ref{thm:dec_pps:ocm}); 
\textit{ii} jointly randomized behavioral and pure decentralized policy-profiles have the same set of occupation measures (Theorem~\ref{thm:dec_pps:ocm_3}); and,  
\textit{iii}) assuming finite number of observations, occupation measures realized from $\vvspace$, $\Mta{1}{\vvspace}$, and $\Mta{1}{\vvspacepure}$ all coincide with those from $\Mta{1}{\uspace}$ and $\Mta{1}{\uspacepure}$ (Theorem~\ref{thm:ciapproach:equivalence}).

Finally, these structural results on occupation measures can be used to develop results on Lagrangian duality, the minimum number of randomizations needed in an optimal behavioral coordination policy, and learning based schemes that can find approximately optimal solutions---see \cite[Chapters 2--5]{nkhanthesis2025} for details. 
